\documentclass[a4paper,12pt,twoside]{amsart}
\usepackage{fullpage}

\usepackage[all]{xy}
\usepackage{mymacros,oitp}

\title[Derived equivalence]
{$G_2$-Grassmannians and
derived equivalences}

\author[K.~Ueda]{Kazushi Ueda}
\address{
Graduate School of Mathematical Sciences,
The University of Tokyo,
3-8-1 Komaba,
Meguro-ku,
Tokyo,
153-8914,
Japan.}
\email{kazushi@ms.u-tokyo.ac.jp}

\date{}
\begin{document}

\maketitle

\begin{abstract}
We prove the derived equivalence
of a pair of non-compact Calabi--Yau 7-folds,
which are the total spaces of certain rank 2 bundles
on $G_2$-Grassmannians.
The proof follows that of the derived equivalence of Calabi--Yau 3-folds
in $G_2$-Grassmannians
by Kuznetsov
\cite{1611.08386}
closely.
\end{abstract}

\section{Introduction}


The simply-connected simple algebraic group $G$ of type $G_2$
has three homogeneous spaces
$\G \coloneqq G/P_1$,
$\Q \coloneqq G/P_2$,
and
$\F \coloneqq G/B$
associated with the crossed Dynkin diagrams
\Fone, \Ftwo, and \Fthree respectively.
The Picard group of $\F$ can be identified with the weight lattice of $G$,
which in turn can be identified with $\bZ^2$ as
$
 (a,b) \coloneqq a \omega_1 + b \omega_2,
$
where $\omega_1$ and $\omega_2$ are the fundamental weights
associated with the long root and the short root respectively.
We write the line bundle associated with the weight $(k,l)$ as $\cO_\F(k,l)$.


Let
\begin{align}
 R
  &\coloneqq \bigoplus_{k,l=0}^\infty H^0 \lb \cO_\F(k,l) \rb
  \cong \bigoplus_{k,l=0}^\infty \lb V^G_{(k,l)} \rb^\dual
\end{align}
be the Cox ring of $\F$,
where $\lb V^G_{(k,l)} \rb^\dual$ is the dual of the irreducible representation of $G$
with the highest weight $(k,l)$.

\begin{comment}


Since $\F$ is a Fano variety,
$R$ is Gorenstein
by \cite[Theorem 1.2]{MR3055212}
(cf.~also \cite[Remark 4.8]{MR3275656}).
Since the canonical bundle $\omega_{\F}$
is isomorphic to $\cO_{\F}(-2,-2)$,
the canonical module $K_R$ is isomorphic
to the shift $R(-2, -2)$ of the free module
(see e.g.~\cite[Lemma 2.12]{MR2641200}).

\end{comment}


The $\bZ^2$-grading of $R$
defines a $(\bGm)^2$-action on $\Spec R$,
which induces an action of $\bGm$
embedded in $(\bGm)^2$
by the anti-diagonal map
$\alpha \mapsto (\alpha,\alpha^{-1})$.
We write the geometric invariant theory quotients as
\begin{align}
 \V_+
  \coloneqq \Proj R_+, \quad
 \V_-
  \coloneqq \Proj R_-, \quad
 \V_0
  \coloneqq \Spec R_0,
\end{align}
where
\begin{align}
 R_n = \bigoplus_{i \in \bZ} R_{i, n-i}, \quad
 R_+ \coloneqq \bigoplus_{n=0}^{\infty} R_n, \quad
 R_- \coloneqq \bigoplus_{n=0}^{\infty} R_{-n}.
\end{align}


$\V_+$ and $\V_-$ are the total spaces
of the dual of the equivariant vector bundles of rank 2
on $\Q$ and $\G$
associated with irreducible representations of $P_1$ and $P_2$
with the highest weight $(1,1)$.
The structure morphisms
$
 \phi_+ \colon \V_+ \to \V_0
$
and
$
 \phi_- \colon \V_- \to \V_0
$
are crepant resolutions
which contract the zero-sections.


The same construction
for the simply-connected simple algebraic group
$\Sp(2)$ of type $C_2$,
which is accidentally isomorphic
to the simply-connected simple algebraic group
$\Spin(5)$
of type $B_2$,
gives the 5-fold flop discussed in \cite{MR3509912},
where it is attributed to Abuaf.


The main result in this paper is the following:

\begin{theorem} \label{th:main}
$\V_+$ and $\V_-$ are derived-equivalent.
\end{theorem}


\pref{th:main} provides an evidence
for the conjecture \cite[Conjecture 4.4]{MR1957019}
\cite[Conjecture 1.2]{MR1949787}
that birationally equivalent smooth projective varieties
are K-equivalent
if and only if they are D-equivalent.


The proof of \pref{th:main} closely follows \cite{1611.08386},
where the derived equivalence of Calabi--Yau complete intersections
in $\G$ and $\Q$
defined by sections of the equivariant vector bundles
dual to $\V_+$ and $\V_-$.
The derived equivalence of these Calabi--Yau 3-folds in turn follows
from \pref{th:main}
using matrix factorizations.


\begin{notations}
We work over a field $\bfk$ throughout this paper.
All pull-back and push-forward are derived.
The complexes underlying $\Ext^\bullet(-,-)$ and $H^\bullet(-)$
will be denoted by $\hom(-,-)$ and $\h(-)$.
\end{notations}

\begin{acknowledgements}
We thank Atsushi Ito,
Makoto Miura, and
Shinnosuke Okawa
for collaborations
\cite{1606.04076,1606.04210,1612.08497}
which led to this work.
K.U. is supported by Grants-in-Aid for Scientific Research
(24740043,
15KT0105,
16K13743,
16H03930).
\end{acknowledgements}

\section{The blow-up diagram}

The $G_2$-Grassmannian
$\G$ is the zero locus
$
 s_\lambda^{-1}(0)
$
of the section
$
 s_\lambda
$
of
the equivariant vector bundle $\cQ^\dual(1)$
of rank 5 on $\Gr(2,V)$,
obtained as the tensor product
of the dual $\cQ^\dual$
of the universal quotient bundle $\cQ$
and the hyperplane bundle $\cO(1)$.
Here $V \coloneqq V^G_{(0,1)}$ is the 7-dimensional fundamental representation of $G_2$, and
$s_\lambda$ corresponds to the $G_2$-invariant 3-form
on $V$
under the isomorphism
$
 H^0(\Gr(2,V),\cQ^\dual(1)) \cong \bigwedge^3 V^\dual.
$
We write the $G_2$-equivariant vector bundle
associated with the irreducible representation of $P_1$
with the highest weight $(a,b)$
as $\cE_{(a,b)}$.
The restriction
$
 \scrU \coloneqq \cS|_\G
$
of the universal subbundle $\cS$ of rank 2 on $\Gr(2,V)$
is isomorphic to $\cE_{(-1,1)}$.

\begin{comment}
It is clear that
$\scrU^\dual$ is a $G_2$-equivariant vector bundle.
Since
\begin{align}
 H^0(\scrU^\dual)
  \cong V^\dual
  \cong \lb V^G_{(0,1)} \rb^\dual
  \cong H^0(\cE_{(0,1)}),
\end{align}
one has
\begin{align}
 \scrU^\dual \cong \cE_{(0,1)}.
\end{align}

By applying
\cite[Theorem 0.2]{MR2238172}
and \cite[Corollary 8.11]{MR2238172}
to the case where
$L=0$, $X_L=\G$ and $Y_L = \emptyset$,
one obtains
a full exceptional collection 
\begin{align} \label{eq:Kuz_col} 
 \lb \cO_{\G}, \scrU^\dual,
 \cO_{\G}(1), \scrU^\dual(1),
 \cO_{\G}(2), \scrU^\dual(2) \rb
\end{align}
in $D^b \coh \G$
\cite[Corollary 8.11]{MR2238172}.
This collection is
\begin{align}
 (\cE_{(0,0)}, \cE_{(0,1)}, \cE_{(1,0)}, \cE_{(1,1)}, \cE_{(2,0)}, \cE_{(2,1)}).
\end{align}
Note that
\begin{align}
 \cE_{(a,1)}^\vee \cong \cE_{(-a-1,1)}
\end{align}
and
\begin{align}
 \omega_{\G} \cong \cE_{(-3,0)}.
\end{align}
The collection is
\begin{align}
\begin{array}{ccc}
 (0,1) & (1,1) & (2,1) \\
 (0,0) & (1,0) & (2,0),
\end{array}
\end{align}
and helix is continued as
\begin{align}
\begin{array}{cccccc}
 \cdots & (0,1) & (1,1) & (2,1) &\cdots \\
 \cdots & (0,0) & (1,0) & (2,0) & (3,0) & \cdots.
\end{array}
\end{align}
One can easily see from
\begin{align}
 \Hom^*(\cE_{(0,1)}, \cE_{(3,0)})
  \cong H^*(\cE_{(-1,1)} \otimes \cE_{(3,0)})
  \cong H^*(\cE_{(2,1)})
\end{align}
and so on
that this helix is strong.
\end{comment}

The $G_2$-flag variety $\F$ is isomorphic to the total space
of the $\bP^1$-bundle
$
 \varpi_+ \colon \bP(\scrU) \to \G
$
associated with $\scrU$
(or any other equivariant vector bundle of rank 2,
since all of them are related by a twist by a line bundle).
We write the relative hyperplane class of $\varpi_+$ as $h$,
so that
\begin{align}
 (\varpi_+)_* \lb \cO_\F(h) \rb \cong \scrU^\dual.
\end{align}
The pull-back to $\F$ of the hyperplane class $H$ in $\G$
will be denoted by $H$ again by abuse of notation.

The other $G_2$-Grassmannian $\Q$
is a quadric hypersurface in $\bP(V)$.
We write the equivariant vector bundle on $\Q$
associated with the irreducible representation of $P_2$
with highest weight $(a,b)$ as $\cF_{(a,b)}$.
The flag variety $\F$ has a structure of a $\bP^1$-bundle
$
 \varpi_- \colon \F \to \Q,
$
whose relative hyperplane class is given by $H$.
We define a vector bundle $\scrK$ on $\Q$ by
\begin{align}
 \scrK \coloneqq \lb (\varpi_-)_* \lb \cO_\F(H) \rb \rb^\dual,
\end{align}
so that
$
 \F \cong \bP_\G(\scrK).
$
One can show that $\scrK$ is isomorphic to $\cF_{(1,-3)}$.
We write the hyperplane class of $\Q$ as $h$
by abuse of notation,
since it pulls back to $h$ on $\F$.

\begin{comment}
On $\F$,
one has
\begin{align}
 \cL_{(1,1)}^\vee
  &\cong \cL_{(-1,-1)}
  \cong \cO_{\F}(-h-H).
\end{align}
On $\Q$,
one has
\begin{align}
 \cF_{(1,1)}^\vee
  &\cong \cF_{(1,-4)}
  \cong \scrK(-h).
\end{align}
On $\G$,
one has
\begin{align}
 \cE_{(1,1)}^\vee
  &\cong \cE_{(-2,1)}
  \cong \scrU(-H).
\end{align}

Recall the diagram
\begin{align} \label{eq:Kuznetsov_diagram}
\begin{gathered}
\xymatrix{
& D \ar@{^(->}[r]^i \ar[ddl]_p & M \ar[d] \ar[ddl]_{\pi_M} \ar[ddr]^{\rho_M} & E \ar@{_(->}[l]_-j \ar[ddr]^q \\
&& \F \ar[dl]^\pi \ar[dr]_\rho \\
X  \ar@{^{(}->}[r] & \Q && \G & Y \ar@{_{(}->}[l] 
}
\end{gathered}
\end{align}
from \cite{1611.08386},
which summarizes the situation in \cite{1607.07821,1606.04210}.
\end{comment}

Let $\V$ be the total space
of the line bundle $\cO_\F(-h-H)$ on $\F$.
The structure morphism
will be denoted by
$
 \pi \colon \V \to \F.
$
The Cox ring of $\V$ is
the $\bN^2$-graded ring
\begin{align}
 S = \bigoplus_{k,l=0}^\infty S_{k,l}
\end{align}
given by
\begin{align}
 S_{k,l}
  &\coloneqq H^0 \lb \cO_\V(k,l) \rb \\
  &\cong H^0 \lb \pi_* \lb \cO_\V(k,l) \rb \rb \\
  &\cong H^0 \lb \pi_* \cO_\V \otimes \cO_\F(k,l) \rb \\
  &\cong H^0 \lb \lb \bigoplus_{m=0}^\infty \cO_\F(m,m) \rb \otimes \cO_\F(k,l) \rb \\
  &\cong \bigoplus_{m=0}^\infty H^0 \lb \cO_\F(k+m,l+m) \rb \\
  &\cong \bigoplus_{m=0}^\infty \lb V^G_{(k+m,l+m)} \rb^\dual,
\end{align}
whose multiple Proj recovers $\V$.
Similarly,
the Cox ring of the total space $\W_+$ of the bundle
$
 \cE_{(1,1)}^\dual \cong \scrU(-H)
$
is given by
$
 \bigoplus_{k=0}^\infty H^0 \lb \cO_{\W_+}(k H) \rb
$
where
\begin{align}
 H^0 \lb \cO_{\W_+}(k H) \rb
  &\cong H^0 \lb \pi_* \lb \cO_{\W_+}(k H) \rb \rb \\
  &\cong H^0 \lb \pi_* \cO_{\W_+} \otimes \cO_\G (k H) \rb \\
  &\cong \bigoplus_{m=0}^\infty H^0 \lb \lb \Sym^m \cE_{(1,1)} \rb \otimes \cO_\G(k H) \rb \\ 
  &\cong \bigoplus_{m=0}^\infty H^0 \lb \cE_{(m,m)} \otimes \cE_{(k,0)} \rb \\ 
  &\cong \bigoplus_{m=0}^\infty H^0 \lb \cE_{(m+k,m)} \rb.
\end{align}
This is isomorphic to $R_+$,
so that $\W_+$ is isomorphic to $\V_+$,
and the affinization morphism
\begin{align} \label{eq:V-affinization}
 \V
  \to \Spec H^0 \lb \cO_\V \rb
  \cong \V_0
\end{align}
is the composition of the natural projection
$
 \varphi_+ \colon \V \to \V_+
$
and the affinization morphism
$
 \phi_+ \colon \V_+ \to \V_0.
$
Since $\V_+$ is the total space of $\cE_{(1,1)}^\dual$,
the ideal sheaf of the zero-section is the image of the natural morphism
from $\pi_+^* \cE_{(1,1)}$ to $\cO_{\V_+}$,
and the morphism $\varphi_+$ is the blow-up along it.
Similarly,
the affinization morphism \pref{eq:V-affinization}
also factors into the blow-up
$
 \varphi_- \colon \V \to \V_-
$
and the affinization morphism
$
 \phi_- \colon \V_- \to \V_0,
$
and one obtains the following commutative diagram:
\begin{align} \label{eq:blow-up_diagram}
\begin{gathered}
\xymatrix{
 & \ar[dl]_{\varphi_+} \V \ar[dr]^{\varphi_-} \\
 \V_+ \ar[dr]^{\phi_+} & & \ar[dl]_{\phi_-} \V_- \\
 & \V_0
}
\end{gathered}
\end{align}

\section{Some extension groups}

The zero-sections and the natural projections fit
into the following diagram:
\begin{align} \label{eq:zero-sections_diagram}
\begin{gathered}
\xymatrix{
 & \F \ar[dl]_{\varpi_+} \ar@{^(->}[d]^\iota \ar[dr]^{\varpi_-} & \\
 \G \ar@{^(->}[d] & \V \ar[dl]_{\phi_+} \ar[dr]^{\phi_-} & \Q \ar@{^(->}[d] \\
 \V+ & & \V_-
}
\end{gathered}
\end{align}
We write
$
 \scrU_\F \coloneqq \varpi_+^* \scrU,
$
$
 \scrS_\F \coloneqq \varpi_-^* \scrS,
$
and
$
 \scrU_\V \coloneqq \pi^* \scrU_\F.
$
By abuse of notation,
we use the same symbol for an object of $\D(\F)$
and its image in $\D(\V)$ by the push-forward $\iota_*$.
Since $\V$ is the total space of $\cO_\V(-h-H)$,
one has a locally free resolution
\begin{align}
 0 \to \cO_\V(h+H) \to \cO_\V \to \cO_\F \to 0
\end{align}
of $\cO_\F$ as an $\cO_\V$-module.

By tensoring $\cO_\F(-h)$ to \cite[Equation (5)]{1611.08386},
one obtains an exact sequence
\begin{align} \label{eq:Ext4}
 0 \to \cO_\F(H-2h) \to \scrU_\F^\dual(-h) \to \cO_\F \to 0.
\end{align}

\pref{lm:Ext1} and \pref{pr:S} below are taken from \cite{1611.08386}:

\begin{lemma}[{\cite[Lemma 1]{1611.08386}}] \label{lm:Ext1}
\begin{enumerate}[(i)]
 \item
Line bundles $\cO_\F(th-H)$ and $\cO_\F(tH-h)$ are acyclic for all $t \in \bZ$.
 \item
Line bundles $\cO_\F(-2H)$ and $\cO_\F(2h-2H)$ are acyclic and
\begin{align*}
 H^\bullet(\cO_\F(3h-2H)) \cong \bfk[-1].
\end{align*}
 \item
Vector bundles $\scrU_\F(-2H)$, $\scrU_\F(-H)$, $\scrU_\F(h-H)$ and $\scrU_\F \otimes \scrU_\F(-H)$
are acyclic,
and
\begin{align*}
 H^\bullet(\scrU_\F(h)) \cong \bfk, \quad
 H^\bullet(\scrU_\F \otimes \scrU_\F(h)) \cong \bfk[-1].
\end{align*}
\end{enumerate}
\end{lemma}

\begin{proposition}[{\cite[Proposition 3 and Lemma 4]{1611.08386}}] \label{pr:S}
One has an exact sequence
\begin{equation} \label{eq:UUExt}
 0 \to \scrU_\F \to \scrS_\F \to \scrU_\F^\dual(-h) \to 0.
\end{equation}
\end{proposition}

\pref{lm:Ext1} immediately implies the following:

\begin{lemma} \label{lm:orth1}
$\cO_\F(-H)$ is right orthogonal to both $\scrU_\F^\dual(-h)$ and $\cO_\F(-h)$.
\end{lemma}

\begin{proof}
We have
\begin{align}
 \hom_{\cO_\V} \lb \cO_\F(-h), \cO_\F(-H) \rb
  &\cong \hom_{\cO_\V} \lb \lc \cO_\V(H) \to \cO_\V(-h) \rc, \cO_\F(-H) \rb \\
  &\cong \h \lb \lc \cO_\F(h-H) \to \cO_\F(-2H) \rc \rb
\end{align}
and
\begin{align}
 \hom_{\cO_\V} \lb \scrU_\F^\dual(-h), \cO_\F(-H) \rb
  &\cong \hom_{\cO_\V} \lb \lc \scrU_\V^\dual(H) \to \scrU_\V^\dual(-h) \rc, \cO_\F(-H) \rb \\
  &\cong \h \lb \lc \scrU_\F(h-H) \to \scrU_\F(-2H) \rc \rb,
\end{align}
both of which vanish by \pref{lm:Ext1}.
\end{proof}

\begin{lemma} \label{lm:Ext2}
One has
\begin{align}
 \hom_{\cO_\V} \lb \scrU_\F^\dual(-h), \scrU_\F \rb
  \cong \bfk[-1].
\end{align}
\end{lemma}

\begin{proof}
One has
\begin{align}
 \hom_{\cO_\V} \lb \scrU_\F^\dual(-h), \scrU_\F \rb
  &\cong \hom_{\cO_\V} \lb \lc \scrU_\V^\dual(H) \to \scrU_\V^\dual(-h) \rc, \scrU_\F \rb \\
  &\cong \h \lb \lc \scrU_\F \otimes \scrU_\F(h) \to \scrU_\F \otimes \scrU_\F(-H) \rc \rb.
\end{align}
\pref{lm:Ext1} shows that
the first term gives $\bfk[-1]$ and the second term vanishes.
\end{proof}

\begin{lemma} \label{lm:Ext3}
One has
\begin{align}
 \hom_{\cO_\V} \lb \scrU_\F^\dual(-h), \cO_\F \rb
  \cong \bfk.
\end{align}
\end{lemma}

\begin{proof}
One has
\begin{align}
 \hom_{\cO_\V} \lb \scrU_\F^\dual(-h), \cO_\F \rb
  &\cong \hom_{\cO_\V} \lb \lc \scrU_\V^\dual(H) \to \scrU_\V^\dual(-h) \rc, \cO_\F \rb \\
  &\cong \h \lb \lc \scrU_\F(h) \to \scrU_\F(-H) \rc \rb.
\end{align}
\pref{lm:Ext1} shows that
the first term gives $\bfk$ and the second term vanishes.
\end{proof}

\begin{lemma} \label{lm:orth2}
One has
\begin{align}
 \hom_{\cO_\V} \lb \cO_\F(H-2h), \cO_\F(h) \rb
  \cong 0.
\end{align}
\end{lemma}

\begin{proof}
One has
\begin{align}
 \hom_{\cO_\V} \lb \cO_\F(H-2h), \cO_\F(h) \rb
  &\cong \hom_{\cO_\V} \lb \lc \cO_\V(2H-h) \to \cO_\V(H-2h) \rc, \cO_\F(h) \rb \\
  &\cong \h \lb \lc \cO_\V(3h-H) \to \cO_\V(2h-2H) \rc \rb,
\end{align}
which vanishes by \pref{lm:Ext1}.
\end{proof}

\section{Derived equivalence by mutation}

Recall from \cite{1611.08386} that
\begin{align} \label{eq:EC1}
 \D(\G) = \la \cO_\G(-H), \scrU, \cO_\G, \scrU^\dual, \cO_\G(H), \scrU^\dual(H) \ra
\end{align}
and
\begin{align} \label{eq:EC2}
 \D(\Q) = \la \cO_\Q(-3h), \cO_\Q(-2h), \cO_\Q(-h), \scrS, \cO_\Q, \cO_\Q(h) \ra.
\end{align}
It follows from \cite{MR1208153} that
\begin{align} \label{eq:SOD1}
 \D(\V) = \la \iota_* \varpi_+^* \D(\G), \Phi_+(\D(\V_+)) \ra
\end{align}
and
\begin{align} \label{eq:SOD2}
 \D(\V) = \la \iota_* \varpi_-^* \D(\Q), \Phi_-(\D(\V_-)) \ra,
\end{align}
where
\begin{align}
 \Phi_+ \coloneqq \phi_+^*(-) \otimes \cO_\V(h) \colon \D(\V_+) \to \D(\V)
\end{align}
and
\begin{align}
 \Phi_- \coloneqq \phi_-^*(-) \otimes \cO_\V(H) \colon \D(\V_-) \to \D(\V).
\end{align}
\pref{eq:EC1} and \pref{eq:SOD1} gives
\begin{align}
 \D(\V) = \la \cO_\F(-H), \scrU_\F, \cO_\F, \scrU_\F^\dual, \cO_\F(H), \scrU_\F^\dual(H), \Phi_+(\D(\V_+)) \ra.
\end{align}
By mutating $\Phi_+(\D(\V_+))$ two steps to the left,
one obtains
\begin{align}
 \D(\V) = \la \cO_\F(-H), \scrU_\F, \cO_\F, \scrU_\F^\dual, \Phi_1(\D(\V_+)), \cO_\F(H), \scrU_\F^\dual(H) \ra
\end{align}
where
\begin{align}
 \Phi_1 \coloneqq \L_{\la \cO_\F(H), \scrU_\F^\dual(H) \ra} \circ \Phi_+.
\end{align}
By mutating the last two terms to the far left,
one obtains
\begin{align}
 \D(\V) = \la \cO_\F(-h), \scrU_\F^\dual(-h), \cO_\F(-H), \scrU_\F, \cO_\F, \scrU_\F^\dual, \Phi_1(\D(\V_+)) \ra,
\end{align}
since $\omega_\V \cong \cO_\V(-h-H)$.
\pref{lm:orth1} allows one to move $\cO_\F(-H)$ to the far left
without affecting other objects:
\begin{align}
 \D(\V) = \la \cO_\F(-H), \cO_\F(-h), \scrU_\F^\dual(-h), \scrU_\F, \cO_\F, \scrU_\F^\dual, \Phi_1(\D(\V_+)) \ra.
\end{align}
By mutating $\scrU_\F$ one step to the left and using \pref{pr:S} and \pref{lm:Ext2},
one obtains
\begin{align}
 \D(\V) = \la \cO_\F(-H), \cO_\F(-h), \scrS_\F, \scrU_\F^\dual(-h), \cO_\F, \scrU_\F^\dual, \Phi_1(\D(\V_+)) \ra.
\end{align}
By mutating $\cO_\F(-H)$ to the far right,
one obtains
\begin{align}
 \D(\V) = \la \cO_\F(-h), \scrS_\F, \scrU_\F^\dual(-h), \cO_\F, \scrU_\F^\dual, \Phi_1(\D(\V_+)), \cO_\F(h) \ra.
\end{align}
By mutating $\Phi_1(\D(\V_+))$ to the right,
one obtains
\begin{align}
 \D(\V) = \la \cO_\F(-h), \scrS_\F, \scrU_\F^\dual(-h), \cO_\F, \scrU_\F^\dual, \cO_\F(h), \Phi_2(\D(\V_+)) \ra
\end{align}
where
\begin{align}
 \Phi_2 \coloneqq \R_{\cO_\F(h)} \circ \Phi_1.
\end{align}
By mutating $\scrU_\F^\dual(-h)$ one step to the right
and using \pref{lm:Ext3} and \pref{eq:Ext4},
one obtains
\begin{align}
 \D(\V) = \la \cO_\F(-h), \scrS_\F, \cO_\F, \cO_\F(H-2h), \scrU_\F^\dual, \cO_\F(h), \Phi_2(\D(\V_+)) \ra.
\end{align}
Similarly,
by mutating $\scrU_\F^\dual$ one step to the right,
one obtains
\begin{align}
 \D(\V) = \la \cO_\F(-h), \scrS_\F, \cO_\F, \cO_\F(H-2h), \cO_\F(h), \cO_\F(H-h), \Phi_2(\D(\V_+)) \ra.
\end{align}
\pref{lm:orth2} allows one to exchange
$\cO_\F(H-2h)$ and $\cO_\F(h)$
to obtain
\begin{align}
 \D(\V) = \la \cO_\F(-h), \scrS_\F, \cO_\F, \cO_\F(h), \cO_\F(H-2h), \cO_\F(H-h), \Phi_2(\D(\V_+)) \ra.
\end{align}
By mutating $\Phi_2(\D(\V_+))$ two steps to the left,
one obtains
\begin{align}
 \D(\V) = \la \cO_\F(-h), \scrS_\F, \cO_\F, \cO_\F(h), \Phi_3(\D(\V_+)), \cO_\F(H-2h), \cO_\F(H-h) \ra
\end{align}
where
\begin{align}
 \Phi_3 \coloneqq \L_{\la \cO_\F(H-2h), \cO_\F(H-h) \ra} \circ \Phi_2.
\end{align}
By mutating the last two terms to the far left,
one obtains
\begin{align} \label{eq:SOD3}
 \D(\V) = \la \cO_\F(-3h), \cO_\F(-2h), \cO_\F(-h), \scrS_\F, \cO_\F, \cO_\F(h), \Phi_3(\D(\V_+)) \ra.
\end{align}
By comparing \pref{eq:SOD3}
with 
\begin{align} \label{eq:SOD4}
 \D(\V) = \la \cO_\F(-3h), \cO_\F(-2h), \cO_\F(-h), \scrS_\F, \cO_\F, \cO_\F(h), \Phi_-(\D(\V_-)) \ra
\end{align}
obtained by combining \pref{eq:EC2} and \pref{eq:SOD2},
one obtains a derived equivalence
\begin{align} \label{eq:Phi}
 \Phi \coloneqq \Phi_-^! \circ \Phi_3 \colon \D(\V_+) \simto \D(\V_-),
\end{align}
where
\begin{align}
 \Phi_-^!(-) \coloneqq (\phi_-)_* \lb (-) \otimes \cO_\V(-H) \rb
  \colon \D(\V) \to \D(\V_-)
\end{align}
is the left adjoint functor of $\Phi_-$.
Note that the left mutation along an exceptional object $\cE \in \D(\V)$
is an integral functor
$
 \Phi_\cK(-) \coloneqq (p_2)_* \lb p_1^*(-) \otimes \cK \rb
$
along the diagram
\begin{align}
\begin{gathered}
\xymatrix{
 & \V \times_{\V_0} \V \ar[dl]_{p_1} \ar[dr]^{p_2} \\
 \V && \V 
}
\end{gathered}
\end{align}
whose kernel $\cK$ is the cone over the evaluation morphism
$
 \ev \colon \cE^\dual \boxtimes \cE \to \Delta_\V.
$
The functors
$
 \Phi_+ \colon \D(\V_+) \to \D(\V)
$
and
$
 \Phi_-^! \colon \D(\V) \to \D(\V_-)
$
are clearly an integral functor,
so that the functor \pref{eq:Phi} is also
an integral functor,
whose kernel is an object of $\D(\V_+ \times_{\V_0} \V_-)$
obtained by convolution.

\begin{comment}
Note that everything is linear over $R_0$.
Any integral functor along any diagram of the form
\begin{align}
\begin{gathered}
\xymatrix{
 & W \ar[dl]_{p_+} \ar[dr]^{p_-} \\
 \V_+ && \V_- 
}
\end{gathered}
\end{align}
can be written as an integral functor along
\begin{align}
\begin{gathered}
\xymatrix{
 & \V_+ \times_{\V_0} \V_- \ar[dl]_{q_+} \ar[dr]^{q_-} \\
 \V_+ && \V_- 
}
\end{gathered}
\end{align}
since one has a diagram
\begin{align}
\begin{gathered}
\xymatrix{
 & W \ar[ddl]_{p_+} \ar[ddr]^{p_-} \ar[d]^r \\
 & \V_+ \times_{\V_0} \V_- \ar[dl]_{q_+} \ar[dr]^{q_-} \\
 \V_+ && \V_- 
}
\end{gathered}
\end{align}
and
\begin{align}
 (p_-)_* (p_+^*(-) \otimes \cK)
  &\cong (q_-)_* r_* \lb r^* q_+^*(-) \otimes \cK \rb \\
  &\cong (q_-)_* \lb q_+^*(-) \otimes r_* \cK \rb.
\end{align}
\end{comment}

\section{Matrix factorizations}

Let $s_+$ be a general section of the equivariant vector bundle
$\cE_{(1,1)}$ on $\G$.
The zero $X_+$ of $s_+$ is a smooth projective Calabi--Yau 3-fold.
Since $\V_+$ is the total space of the dual bundle
$\cE_{(1,1)}^\dual$ on $\G$,
the space of regular functions on $\V_+$
which are linear along the fiber can naturally be identified
with the space of sections of $\cE_{(1,1)}$.
We write the regular function on $\V_+$
associated with $s_+ \in H^0 \lb \cE_{(1,1)} \rb$
as $\varsigma_+ \in H^0 \lb \cO_{\V_+} \rb$.
The zero $D_+$ of $\varsigma_+$ is the union
of a line sub-bundle of $\V_+$
and the inverse image of $X_+$
by the structure morphism $\pi_+ \colon \V_+ \to \G$.
The singular locus of $D_+$ is given by $X_+$.

Let $\varsigma_-$ be a regular function on $\V_-$
corresponding to $\varsigma_+$
under the isomorphism
$
 H^0 \lb \cO_{\V_+} \rb
  \cong H^0 \lb \cO_{\V_0} \rb
  \cong H^0 \lb \cO_{\V_-} \rb
$
given by the diagram in \pref{eq:blow-up_diagram},
and $X_-$ be the zero of the corresponding section
$
 s_- \in H^0 \lb \cF_{(1,1)} \rb,
$
which is a smooth projective Calabi--Yau 3-fold in $\Q$.

The push-forward of the kernel of $\Phi$ on $\V_+ \times_{V_0} \V_-$
to $\V_+ \times_{\bA^1} \V_-$
gives a kernel of $\Phi$ on $\V_+ \times_{\bA^1} \V_-$.
By taking the base-change along the inclusion $0 \to \bA^1$ of the origin
and applying \cite[Proposition 2.44]{MR2238172},
one obtains an equivalence
$
 \Phi_0 \colon \D(D_+) \cong \D(D_-)
$
of the bounded derived categories of coherent sheaves.
By using either of the characterization of perfect complexes
as \emph{homologically finite} objects
(i.e., objects whose total $\Ext$-groups with any other object
are finite-dimensional)
or \emph{compact} objects
(i.e., objects such that the covariant functors represented by them
commute with direct sums),
one deduces that $\Phi_D$ preserves perfect complexes,
so that it induces an equivalence
$
 \Phi_0^\sing \colon \Dsg(D_+) \cong \Dsg(D_-)
$
of singularity categories
(see \cite[Section 7]{MR2563433} and
\cite[Theorem 1.1]{MR2593258}).

Recall that $\V_+$, $\V_-$ and $\V_0$ are geometric invariant theory quotient
of $\Spec R$
by the anti-diagonal $\bGm$-action.
The residual diagonal $\bGm$-action on both $\V_+$ and $\V_-$ are
dilation action on the fiber.
The equivalences $\Phi$, $\Phi_0$ and $\Phi_0^\sing$
are equivariant with respect to this $\bGm$-action,
and induces an equivalence of $\bGm$-equivariant categories
\cite[Theorem 1.1]{1506.00177},
which will be denoted by the same symbol by abuse of notation.
Now \cite[Theorem 3.6]{MR3071664} gives equivalences
\begin{align} \label{eq:Isik1}
 \Dsg([D_+/\bGm]) \cong \D(X_+)
\end{align}
and
\begin{align} \label{eq:Isik2}
 \Dsg([D_-/\bGm]) \cong \D(X_-)
\end{align}
between $\bGm$-equivariant singularity categories
and derived categories
of coherent sheaves
(see also \cite{MR2982435}
where the case of line bundles is discussed
independently and around the same time as \cite{MR3071664}).
By composing these derived equivalences
with $\Phi_0^\sing$,
one obtains a derived equivalence
between $X_+$ and $X_-$.
It is an interesting problem to compare this equivalence
with the one obtained in \cite{1611.08386}.
Another interesting problem is to prove the derived equivalence
using variation of geometric invariant theory quotient
along the lines of \cite{0803.2045,MR2795327,1203.6643,MR3327537,MR3370126},
and use it to produce autoequivalences of the derived category
\cite{MR3223878,MR3552550}.

\bibliographystyle{amsalpha}
\bibliography{bibs}

\end{document}